\documentclass[a4paper, 12pt, onecolumn]{article} \textwidth 148mm
\usepackage{amsmath}
\usepackage{amsthm}
\usepackage{amsfonts}
\usepackage{bbm}
\usepackage{CJK}
\usepackage{fancyhdr}
\usepackage{graphicx}
\usepackage{geometry}
\usepackage{psfrag}
\usepackage{amsfonts,amsmath,amsthm, amssymb}
\usepackage{latexsym, euscript, epic, eepic}
\usepackage{time}
\usepackage{txfonts}
\usepackage{colortbl}
\usepackage{stmaryrd}
\usepackage{mathrsfs}
\usepackage{txfonts}
\usepackage{amsfonts}
\usepackage{color}
\usepackage{lineno}

\usepackage{indentfirst,latexsym,bm}

 \numberwithin{equation}{section}
\begin{document}
\newtheorem{theorem}{Theorem}[section]
\newtheorem{proposition}[theorem]{Proposition}
\newtheorem{remark}[theorem]{Remark}
\newtheorem{corollary}[theorem]{Corollary}
\newtheorem{definition}{Definition}[section]
\newtheorem{lemma}[theorem]{Lemma}

\title{\large  Measure Upper Bounds of Nodal Sets of Robin Eigenfunctions}
\author{\normalsize Fang Liu \\
\scriptsize Department of Mathematics, School of Science, Nanjing
University of Science and Technology, Nanjing 210094, China
\\\scriptsize Email: sdqdlf78@126.com
\\\normalsize Long Tian\footnote{: Corresponding
author.} \\
\scriptsize Department of Mathematics, School of Science, Nanjing
University of Science
and Technology, Nanjing 210094, China\\
\scriptsize Email: tianlong19850812@163.com\\
\normalsize Xiaoping Yang \\
\scriptsize Department of Mathematics, Nanjing University, Nanjing
210093, China
\\
\scriptsize Email: xpyang@nju.edu.cn
\\
\scriptsize This work is supported by National Natural Science
Foundation of China (No. 11401307, 11501292 and 11531005).}
\date{}
\maketitle
 \fontsize{12}{22}\selectfont\small
 \paragraph{Abstract:} In this paper, we obtain the upper bounds for the Hausdorff measures of nodal sets of eigenfunctions with the Robin boundary conditions, i.e.,
\begin{equation}\nonumber
{\left\{ \begin{array}{l}
 \triangle u+\lambda u=0,\quad in\quad
\Omega,\\
u_{\nu}+\mu u=0,\quad on\quad\partial\Omega,
\end{array} \right.}
\end{equation}
%
  where the domain $\Omega\subseteq\mathbb{R}^n$, $u_{\nu}$ means the derivative of $u$ along the outer normal direction of $\partial\Omega$.
  We show that, if $\Omega$ is bounded and analytic, and the corresponding eigenvalue $\lambda$ is large enough,
  then the measure upper bounds for the nodal sets of eigenfunctions are $C\sqrt{\lambda}$,
  where $C$ is a positive constant depending only on $n$ and $\Omega$ but not on $\mu$.
  We also show that, if $\partial\Omega$ is $C^{\infty}$ smooth and $\partial\Omega\setminus\Gamma$ is piecewise analytic,
  where $\Gamma\subseteq\partial\Omega$ is a union of some $n-2$ dimensional submanifolds of $\partial\Omega$, $\mu>0$, and $\lambda$ is large enough,
  then the corresponding measure upper bounds for the  nodal sets of $u$ are $C(\sqrt{\lambda}+\mu^{\alpha}+\mu^{-c\alpha})$ for some positive number $\alpha$,
  where $c$ is a positive constant depending only on $n$, and $C$ is a positive constant depending on $n$, $\Omega$, $\Gamma$ and $\alpha$.
 \\[10pt]
\\[10pt]

\vspace{1cm}\fontsize{12}{22}\selectfont

\section{Introduction}
In this paper, we focus on the following eigenfunctions with Robin boundary condition
\begin{equation}\label{basic equation of eigenfunction}
{\left\{ \begin{array}{l}
 \triangle u+\lambda u=0,\quad in\quad
\Omega,\\
u_{\nu}+\mu u=0,\quad on\quad\partial\Omega,
\end{array} \right.}
\end{equation}
where $\Omega\subseteq\mathbb{R}^n$ is a $C^{\infty}$ bounded
domain, $\mu$ is a constant, $\nu$ is the outer unit normal vector
to $\partial\Omega$, and $\lambda>0$ is the corresponding
eigenvalue. The nodal set of $u$ means the zero level set of $u$,
i.e., the set $\left\{x\in\Omega| u(x)=0\right\}$.

In $\cite{S.T.Yau}$, S.T. Yau conjectured that measures of nodal
sets of the eigenfunctions on compact $C^{\infty}$ Riemannian
manifolds without boundary are comparable to $\sqrt{\lambda}$. In
$\cite{Donnelly and Fefferman}$, H. Donnelly and C. Fefferman proved
Yau's conjecture in the real analytic case. For the non analytic case, in \cite{R.T.Dong},
R.T. Dong showed that an upper bound for the measures of nodal sets
of eigenfunctions for two dimensional case was
$C\lambda^{\frac{3}{4}}$ and proposed an interesting argument to
deal with this problem. A. Logunov in \cite{A.Logunov2} improved the
upper bound for the measures of nodal sets of eigenfunctions
for any dimensional compact Riemannian manifolds to
$C\lambda^\alpha$ for some constant $\alpha
> \frac{1}{2}.$
In 1991, F.H. Lin in $\cite{F.H.Lin}$ investigated the measure
estimates of nodal sets of solutions to uniformly linear
elliptic equations of second order with analytic coefficients by using the the frequency
function which was introduced in \cite{Almgren} in 1979 and also
gave the measure upper bound for the nodal sets of eigenfunctions for the analytic case.
In 1994, Q. Han and F.H. Lin in $\cite{Q.Han F.H.Lin}$
showed upper measure estimates of nodal sets of solutions
to uniformly linear elliptic equations of second order with $C^{1,1}$ coefficients.

The lower bound of the conjecture for two dimensional surfaces was
proven by J.$Br\ddot{u}ning$ in $\cite{J.Bruning}$ and by S.T. Yau
independently.  In the past decades, there are a lot of work
concerning this problem. C.D. Sogge and S. Zelditch in
$\cite{C.D.Sogge and S.Zelditch}$ proved that a lower bound for the
measures of nodal sets of eigenfunctions on compact $C^{\infty}$
Riemannian manifolds was $C\lambda^{(7-3n)/8}$. In
$\cite{T.H.Colding and W.P.Minicozzi}$, T.H. Colding and W.P.
Minicozzi showed that such a lower bound can be improved to be
$C\lambda^{(3-n)/4}$. Recently, A. Logunov in $\cite{A.Logunov2}$
proved the Yau's conjecture of the lower bound. For the
related research work, we refer, for example, $\cite{R.Hardt and L.Simon}$, $\cite{K.Bellova and F.H.Lin}$, $\cite{Long Tian and Xiaoping Yang}$,
$\cite{Long Tian and Xiaoping Yang2}$, $\cite{A.Logunov and E.malinnikova}$, $\cite{Q.Han R.Hardt F.H.Lin}$, $\cite{J.Y.Zhu}$, $\cite{J.E.Chang}$.

In $\cite{I.Kukavica}$, I. Kukavica studied a class of general elliptic linear operator of order 2m and proved that $C\lambda^{\frac{1}{2m}}$ is the
upper measure bound for the nodal set of an eigenfunction $u$ satisfying the boundary condition $B_{j}u=0$ on $\partial\Omega$, where $B_j$ $(j=1, 2, ... , m-1)$ is a linear boundary differential operator,
provided that $\Omega$ is a bounded, analytic domain in
$\mathbb{R}^n$. In fact, the conclusion indicates that a measure upper bound for the nodal set of  an eigenfunction $u$ with the Robin
boudary condition is also $C\lambda^{\frac{1}{2m}}$, but one only knows that the corresponding constant $C$ depends on $n$, $\Omega$ and $\mu$. In $\cite{S.Ariturk}$, S. Ariturk gave some lower
bounds for measures of nodal sets of eigenfunctions on smooth
Riemannian manifolds with Dirichlet or Neumann boundary condition.

 Note that
for the eigenvalue problem $(\ref{basic equation of eigenfunction})$
when $\mu\rightarrow0$, it tends to be the Neumann boundary
condition; and when $\mu\rightarrow\infty$, it tends to be the
Dirichlet boundary condition. Thus one may expect  that the upper
bound for the measure of the  nodal set of the eigenfunction $u$
is also $C\sqrt{\lambda}$, where $C$ is independent of $\mu$,
provided that $\Omega$ is a bounded analytic domain. In this paper,
we will first show that this  is true. More precisely,  we have the
following result.
\begin{theorem}\label{main theorem}
If $\Omega$ is a bounded analytic domain in $\mathbb{R}^n$, then the
upper bound for the Hausdorff measure of the nodal set of an
eigenfunction $u$ is
\begin{equation}
\mathcal{H}^{n-1}\left(\left\{x\in\Omega|u(x)=0\right\}\right)\leq C\sqrt{\lambda},
\end{equation}
where $C$ is a positive constant depending  on $n$ and $\Omega$, but
independent of $\mu$.
\end{theorem}

For the non-analytic case, we get the following conclusion.

\begin{theorem}\label{another main theorem}
Suppose that $\Omega$ is a bounded $C^{\infty}$ domain in
$\mathbb{R}^n$ and $\partial\Omega\setminus\Gamma$ is piecewise
analytic, where $\Gamma\subseteq\partial\Omega$ is a finite union of
some $(n-2)$ dimensional submanifolds. Then if $\lambda$ is large
enough and $\mu>0$, it holds that for any $\alpha\in(0,1)$,
\begin{equation}
\mathcal{H}^{n-1}\left(\left\{x\in\Omega|u(x)=0\right\}\right)\leq
C\left(\sqrt{\lambda}+\mu^{\alpha}+\mu^{-c\alpha}\right),
\end{equation}
where $C$ is a positive constant depending only on $n$, $\Omega$, $\Gamma$ and $\alpha$, and $c$ is a positive constant depending only on $n$.
\end{theorem}

The rest of this paper is organized as follows. In
section 2, we consider the case that $\partial\Omega$ is analytic. We
first show that the upper bound for the Hausdorff measure of the
nodal set of $u$ is $C(\sqrt{\lambda}+|\mu|)$. Then we show that for
$|\mu|$ large enough, the corresponding upper bound is
$C\sqrt{\lambda}$.  We first extend the eigenfunction $u$ into a
neighborhood of $\Omega$. Then we establish two different estimates
for the extended function $u$ for small and large $|\mu|$
respectively. Finally, we get the upper Hausdorff measure of the
nodal set of $u$ in $\Omega$ by the similar argument in
$\cite{I.Kukavica}$.
In section 3, we focus on the case that
$\partial\Omega$ is not analytic.  By using the iteration argument
developed in $\cite{Long Tian and Xiaoping Yang}$, we first give
the upper bound for the doubling index of $u$ away from the
non-analytic part $\Gamma$, and further control the doubling index
near the non-analytic part $\Gamma$. We would like to point out that
such an upper bound of the doubling index probably go to infinity
when the center of the doubling index tends to $\Gamma$.
Fortunately, with the help of the fact
the dimension of $\Gamma$ is $n-2$,
we can control the upper bound and get the desired result.


\section{The Analytic Case}
In order to get a measure upper bound for the nodal set of $u$, we
first need to extend $u$ into a neighborhood of $\Omega$.

\begin{lemma}\label{extending}
Let $u$ be an eigenfunction of $(\ref{basic equation of
eigenfunction})$
and $\lambda$ is the corresponding eigenvalue. Suppose that
$\partial\Omega$ is analytic and $\lambda$ is large enough. Then $u$
can be analytically  extended into $\Omega_{\delta}$, where
$\Omega_{\delta}$ denotes the $\delta$ neighborhood of $\Omega$ for
some $\delta>0$ depending only on $n$ and $\Omega$, such that
\begin{equation}
\|u\|_{L^{\infty}(\Omega_{\delta})}\leq e^{C(\sqrt{\lambda}+|\mu|)}\|u\|_{L^{\infty}(\Omega)},
\end{equation}
where $C$ is a positive constants depending only on $n$ and $\Omega$.

Moreover, if $|\mu|$ is large enough, i.e., $|\mu|>C_0$ for some
positive constant $C_0$ which depends only on $n$ and $\Omega$, then
it holds that
\begin{equation}
\|u\|_{L^{\infty}(\Omega_{\delta})}\leq e^{C\sqrt{\lambda}}\|u\|_{L^{\infty}(\Omega)}.
\end{equation}
Here $C$ and $\delta$ are also positive constants depending only on
$n$ and $\Omega$.
\end{lemma}

\begin{proof}
By the standard elliptic estimates, we first know that
\begin{equation*}
\|u\|_{W^{k,2}(\Omega)}\leq C^k\left(\lambda^{\frac{k}{2}}\|u\|_{L^2(\Omega)}+\|u\|_{W^{k,2}(\partial\Omega)}\right).
\end{equation*}
Now we only need to consider the estimate of
$\|u\|_{W^{k,2}(\partial\Omega)}$.

For any fixed point $x_0\in\partial\Omega$, we make a suitable
transformation, such that $x_0=0$, $\nu$ is the opposite direction
of the axis $x_n$, and the hyperplane $x_n=0$ is the tangent plane
of $\partial\Omega$ at $x_0$.
For any point $x$ near the origin point $x_0$, one may assume that $u_{x_n}=\overrightarrow{\gamma}\cdot u_{\tau}+\beta u_{\nu}$, where $\tau$ is the tangent vector field on $\partial\Omega$, $\overrightarrow{\gamma}$ is an $n-1$ dimensional vector valued function. Because $\overrightarrow{\gamma}$ and $\beta$ both depend only on $\Omega$, they are also analytic functions. So
\begin{eqnarray*}
u_{x_nx_n}&=&\overrightarrow{\gamma}\cdot(u_{x_n})_{\tau}+\beta(u_{x_n})_{\nu}
\\&=&\overrightarrow{\gamma}\cdot(\overrightarrow{\gamma}_{\tau}\cdot u_{\tau})+\overrightarrow{\gamma}\cdot(\overrightarrow{\gamma}\cdot u_{\tau\tau})+\overrightarrow{\gamma}\cdot\beta_{\tau}u_{\nu}+2\beta\overrightarrow{\gamma}\cdot u_{\nu\tau}
\\&+&\beta\overrightarrow{\gamma}_{\nu}\cdot u_{\tau}+\beta\beta_{\nu}u_{\nu}+\beta^2u_{\nu\nu}.
\end{eqnarray*}
It is also obvious that $\overrightarrow{\gamma}(0)=0$ and $\beta(0)=-1$. Then from the Robin boundary condition, we have
\begin{equation*}
u_{x_nx_n}(0)=\mu^2u(0)-\mu\beta_{\nu}(0)u(0)-\overrightarrow{\gamma}(0)u_{\tau}(0).
\end{equation*}
From $\cite{I.Kukavica}$, or the standard elliptic estimate, we know that $u$ is analytic in $\overline{\Omega}$. Thus from $u$ satisfies the equation $\triangle u+\lambda u=0$ in $\Omega$, we have that $u$ satisfies the same equation on $\overline{\Omega}$ and thus it holds that $\triangle u(0)+\lambda u(0)=0$. So from the above calculation, we obtain that
\begin{equation*}
u_{x_1x_1}+u_{x_2x_2}+\cdots+u_{x_{n-1}x_{n-1}}-\overrightarrow{\gamma}u_{\tau}
+(\lambda+\mu^2-\beta_{\nu}\mu)u=0
\end{equation*}
holds on the origin point $x_0$.
Because $x_0$ is an arbitrary point on $\partial\Omega$, the equation
$\triangle u+\lambda u=0$ becomes
\begin{equation}\label{eigen equation on boundary}
\triangle_{\partial\Omega}u+<\vec{b},\nabla_{\partial\Omega}u>+(\lambda+\mu^2+c\mu)u=0
\end{equation}
on the submanifold $\partial\Omega$. Here $\vec{b}$ and $c$
are coefficient functions depending only on $n$ and $\Omega$, $\nabla_{\partial\Omega}$ is the gradient operator on the submanifold $\partial\Omega$, $<,>$ is the inner product on the submainfold $\partial\Omega$, and $\triangle_{\partial\Omega}$ is the
Laplacian on the submanifold $\partial\Omega$.
Because $\partial\Omega$ is an $n-1$ dimensional analytic compact
manifold, the coefficient functions $\vec{b}$ and $c$ both
are analytic on $\partial\Omega$. Thus by the standard elliptic
estimates on Riemannian manifolds, we have the following estimate:
\begin{equation*}
\|u\|_{W^{k,2}(\partial\Omega)}\leq C^k(\lambda+\mu^2)^{\frac{k}{2}}\|u\|_{L^2(\partial\Omega)}.
\end{equation*}
Thus
\begin{eqnarray*}
\|u\|_{W^{k,2}(\Omega)}&\leq&C^k(\sqrt{\lambda}+|\mu|)^k(\|u\|_{L^2(\Omega)}+\|u\|_{L^2(\partial\Omega)})
\\&\leq&C^k(\sqrt{\lambda}+|\mu|)^k\|u\|_{L^{\infty}(\Omega)}.
\end{eqnarray*}
By the Sobolev Embedding Theorem, we have that
\begin{equation*}
||D^{\alpha}u||_{L^{\infty}(\Omega)}\leq
C^{|\alpha|+\frac{n+2}{2}}(\sqrt{\lambda}+|\mu|)^{|\alpha|+\frac{n+2}{2}}\|u\|_{L^{\infty}(\Omega)}.
\end{equation*}
Thus we can extend $u$ into a neighborhood of $\Omega$ by the Taylor
series, noted by $\Omega_{\delta}$, such that
\begin{eqnarray*}
\|u\|_{L^{\infty}(\Omega_{\delta})}&\leq&
\|u\|_{L^{\infty}(\Omega)}+
\sum\limits_{k=1}^{\infty}\sum\limits_{|\alpha|=k}\frac{\delta^k}{\alpha!}\|D^{\alpha}u\|_{L^{\infty}(\Omega)}
\\&\leq&\|u\|_{L^{\infty}(\Omega)}\left(1+\sum\limits_{k=1}^{\infty}\sum\limits_{|\alpha|=k}
\frac{(C\delta)^k(\sqrt{\lambda}+|\mu|)^{k+\frac{n+2}{2}}}{\alpha!}\right)
\\&\leq&e^{C(\sqrt{\lambda}+|\mu|)}\|u\|_{L^{\infty}(\Omega)}
\end{eqnarray*}
for $\lambda$ large enough. That is the desired result.

On the other hand, by the standard elliptic estimate, the Sobolev
Embedding Theorem, and the Robin boundary condition, it also
holds that
\begin{eqnarray*}
\|D^{\alpha}u\|_{L^{\infty}(\Omega)}
&\leq&C^{|\alpha|+\frac{n+2}{2}}\left(\lambda^{\frac{|\alpha|}{2}+
\frac{n+2}{4}}\|u\|_{L^2(\Omega)}+
\|D^{\alpha}u\|_{L^{\infty}(\partial\Omega)}\right)
\\&\leq&
C^{|\alpha|+\frac{n+2}{2}}\left(\lambda^{\frac{|\alpha|}{2}+
\frac{n+2}{4}}\|u\|_{L^2(\Omega)}+
\frac{1}{|\mu|}\|(D^{\alpha}u)_{\nu}\|_{L^{\infty}(\partial\Omega)}\right).
\end{eqnarray*}
Thus for any $x_0\in\partial\Omega$ and $x\in \partial B_r(x_0)$
with $r\leq\delta$, it holds that
\begin{eqnarray*}
\sum\limits_{k=0}^{\infty}\sum\limits_{|\alpha|=k}\frac{(x-x_0)^{\alpha}}{\alpha!}|D^{\alpha}u(x_0)|
&\leq&\sum\limits_{k=0}^{\infty}\sum\limits_{|\alpha|=k}\frac{C^kr^k}{\alpha!}\|D^{\alpha}u\|_{L^{\infty}(\Omega)}
\\&\leq&\sum\limits_{k=0}^{\infty}\sum\limits_{|\alpha|=k}
\frac{C^kr^k}{\alpha!}\left(\lambda^{\frac{k}{2}+\frac{n+2}{4}}\|u\|_{L^2(\Omega)}
+\frac{1}{|\mu|}\|(D^{\alpha}u)_{\nu}\|_{L^{\infty}(\partial\Omega)}\right)
\\&\leq&\sum\limits_{k=0}^{\infty}\sum\limits_{|\alpha|=k}
\frac{C^kr^k}{\alpha!}\left(\lambda^{\frac{k}{2}+\frac{n+2}{4}}\|u\|_{L^{\infty}(\Omega)}
+\frac{1}{|\mu|}\|(D^{\alpha}u)_{\nu}\|_{L^{\infty}(\Omega)}\right)
\\&\leq&e^{C\sqrt{\lambda}r}\|u\|_{L^{\infty}(\Omega)}
+\frac{C}{|\mu|r}\sum\limits_{k=0}^{\infty}\sum\limits_{|\alpha|=k}
\frac{C^kr^k}{\alpha!}\|D^{\alpha}u\|_{L^{\infty}(\Omega)}.
\end{eqnarray*}
In the above inequalities, $C$ may be different from line to line.

Noticing that for $|\mu|$ large enough, the coefficient
$\frac{C}{|\mu|r}$ on the second term of the last inequality can be
controlled by $\frac{1}{2}$, then the
following estimate holds for any $x\in\partial B_r(x_0)$,
\begin{eqnarray*}
|u(x)|&\leq&\sum\limits_{k=0}^{\infty}\sum\limits_{|\alpha|=k}
\frac{C^kr^k}{\alpha!}\|D^{\alpha}u\|_{L^{\infty}(\Omega)}
\\&\leq&e^{C\sqrt{\lambda}r}\|u\|_{L^{\infty}(\Omega)}.
\end{eqnarray*}
Thus by the standard elliptic estimates, we know that for any $x\in
B_r(x_0)$, there holds
\begin{equation*}
|u(x)|\leq e^{C\sqrt{\lambda}r}\|u\|_{L^{\infty}(\Omega)}.
\end{equation*}
Put $r=\delta$, we can get the desired result, provided that $|\mu|\geq\frac{C}{\delta}$.
\end{proof}

Now we adopt the quantity $N(x_0,r)$ as follows:
\begin{equation}
N(x_0,r)=\log_2\frac{\max\limits_{x\in B_r(x_0)}|u(x)|}{\max\limits_{x\in B_{\frac{r}{2}}(x_0)}|u(x)|}.
\end{equation}
It is called in $\cite{A.Logunov}$ the doubling index of $u$
centered at $x_0$ with radius $r$.

We give the upper bound for the doubling index as follows.

\begin{lemma}\label{upper bound for the doubling index}
Let $u$ be an eigenfunction in $\Omega$ and $\lambda$ be the corresponding eigenvalue. Then it holds that
\begin{equation}
N(x,r)\leq C(\sqrt{\lambda}+|\mu|),
\end{equation}
with $x\in\Omega$ and $r\leq\delta$, $\delta$ is a positive constant as in Lemma $\ref{extending}$, and $C$ is a positive constant depending only on $n$ and $\Omega$.

Moreover, if $|\mu|$ is large enough, i.e., $|\mu|\geq M$ for some positive constant $M$ depending only on $n$ and $\Omega$, then
\begin{equation}
N(x,r)\leq C\sqrt{\lambda},\quad\forall x\in\Omega,\quad\forall r\leq\delta,
\end{equation}
where $C$ is a positive constant depending only on $n$ and $\Omega$.
\end{lemma}

\begin{proof}
Let $\overline{x}$ be the maximum point of $u$ at $\overline{\Omega}$. Then for $r\leq\delta$,
\begin{equation}
\begin{array}{l}
\|u\|_{L^{\infty}(B_r(\overline{x}))} \leq
\|u\|_{L^{\infty}(\Omega_{\delta})}
\\ \quad\quad\quad\quad\;\ \leq e^{C(\sqrt{\lambda}+|\mu|)}\|u\|_{L^{\infty}(\Omega)}
\\ \quad\quad\quad\quad\;\ \leq e^{C(\sqrt{\lambda}+|\mu|)}|u(\overline{x})|
\\ \quad\quad\quad\quad\;\ \leq e^{C(\sqrt{\lambda}+|\mu|)}\|u\|_{L^{\infty}(B_{r/2}(\overline{x}))}.
\end{array}
\end{equation}
Thus by the definition of the doubling index, we have
\begin{equation*}
N(\bar{x},r)\leq C(\sqrt{\lambda}+|\mu|)
\end{equation*}
for any $r\leq\delta$. Noting that for any $x_1\in
B_{r/4}(\overline{x})$,
\begin{equation*}
\|u\|_{L^{\infty}(B_{r/2}(x_1))}\geq\|u\|_{L^{\infty}(B_{r/4}(\overline{x}))}\geq\|u\|_{L^{\infty}(\Omega)},
\end{equation*}
and
\begin{equation*}
\|u\|_{L^{\infty}(B_r(\overline{x_1}))}\leq e^{C(\sqrt{\lambda}+|\mu|)}\|u\|_{L^{\infty}(\Omega)},
\end{equation*}
we have
\begin{equation*}
N(x_1,r)\leq C(\sqrt{\lambda}+|\mu|)
\end{equation*}
 for any $r\leq\delta$ and $x_1\in B_{r/4}(\overline{x})$. For
any $x$ in $\Omega$, taking $r=\delta$, we may use the above
arguments for at most $k$ times, where $k$ is a positive constant
depending only on $n$ and $\Omega$, to get that the upper bound for
the doubling index is
\begin{equation*}
N(x,\delta)\leq C(\sqrt{\lambda}+|\mu|).
\end{equation*}
For the radius $r<\delta$, we can use the almost monotonicity
formula, which is stated in $\cite{A.Logunov}$, and is also
implicitly stated in $\cite{Q.Han and F.H.Lin}$, to get the desired
estimate.

If $|\mu|$ is large enough, by Lemma $\ref{extending}$, the above
arguments also hold only if   the quantity $\sqrt{\lambda}+|\mu|$ is
 replaced  by $\sqrt{\lambda}. $
\end{proof}

\begin{remark}
In fact, the above lemma tells us that, for any $x\in\Omega$ and
$r\leq\delta$, it holds that
\begin{equation}
N(x,r)\leq C\sqrt{\lambda},
\end{equation}
where $C$ is a positive constant depending only on $n$ and $\Omega$.
\end{remark}

Before we give the proof of Theorem $\ref{main theorem}$, we state
the following lemma, which is stated in $\cite{Q.Han and
F.H.Lin}$, see also $\cite{A.Logunov,Donnelly and Fefferman}$.
For the completeness, we also give a sketch of its proof.

\begin{lemma}\label{nodal estimate in small ball}
Let $u$ be an analytic function in $B_r(x_0)$. Then it holds that
\begin{equation}
\mathcal{H}^{n-1}\left(\left\{x\in B_{r/16}(x_0)|u(x)=0\right\}\right)\leq CNr^{n-1},
\end{equation}
where $N=\max\left\{N(x,\rho),x\in B_{r/2}(x_0),\rho\leq r/2\right\}$, and $C$ is a positive constant depending only on $n$.
\end{lemma}

\begin{proof}
Without loss of generality, we may assume that $\|u\|_{L^{\infty}(B_r(x_0))}=1$. Then from the assumption of $N$, we know that for any $p\in B_{r/4}(x_0)$, it holds that $\|u\|_{L^{\infty}(B_{1/16}(p))}\geq4^{-cN}$, where $c$ is a positive constant depending only on $n$. So there exists some point $x_{p}\in B_{r/16}(p)$ such that $|u(x_p)|\geq2^{-cN}$ Choose $p_j$ be the points on $\partial B_{r/4}(x_0)$ with $p_j$ on the $x_j$ axis, $j=1,2,\cdots,n$. Let $x_{p_j}\in B_{r/16}(p_j)$ be the points such point
$|u(x_{p_j})|\geq2^{-cN}$. For each $j=1,2,\cdots,n$ and $\omega$ on the unit sphere, let $f_{j,\omega}(t)=u(x_{p_j}+tr\omega)$ for $t\in(-5/8,5/8)$. Then $f_{j,\omega}(t)$ is analytic for $t\in(-5/8,5/8)$. So we can extend $f_{j,\omega}(t)$ to an analytic function $f_{j,\omega}(z)$ for $z=t+iy$ with $|t|<5/8$ and $|y|<y_0$ for some positive number $y_0$. Then we have $|f_{j,\omega}(0)|\geq 2^{-cN}$ and $|f_{j,\omega}(z)|\leq 1$. Applying Lemma 2.3.2 in $\cite{Q.Han and F.H.Lin}$, such a conclusion can also be seen in $\cite{F.H.Lin,Donnelly and Fefferman}$, we have that
\begin{equation*}
\sharp\left\{t|u(x_{p_j}+tr\omega)=0,|t|<1/2\right\}\leq CN.
\end{equation*}
In particular, it holds that
\begin{equation*}
\sharp\left\{t|u(x_{p_j}+tr\omega)=0,x_{p_j}+tr\omega\in B_{r/16}(x_0)\right\}\leq CN.
\end{equation*}
Then from the integral geometric formula, which can be seen in $\cite{F.H.Lin and X.P.Yang,H.Federer}$, we have
\begin{equation*}
\mathcal{H}^{n-1}\left(\left\{x\in B_{r/16}(x_0)|u(x)=0\right\}\right)\leq CNr^{n-1}.
\end{equation*}
\end{proof}

By the above lemmas, we can get the conclusion of Theorem $\ref{main
theorem}$.

\textbf{Proof of Theorem \ref{main theorem}.}

Using Lemmas $\ref{nodal estimate in small ball}$ and $\ref{upper
bound for the doubling index}$, we have
\begin{equation*}
\mathcal{H}^{n-1}\left(\left\{x\in\Omega|u(x)=0\right\}\right)\leq C(\sqrt{\lambda}+|\mu|).
\end{equation*}
On the other hand, for $|\mu|$ large enough, Lemma $\ref{nodal estimate in small ball}$ and Lemma $\ref{upper bound for the doubling index}$ show that
\begin{equation*}
\mathcal{H}^{n-1}\left(\left\{x\in\Omega|u(x)=0\right\}\right)\leq C\sqrt{\lambda}.
\end{equation*}

\qed


\section{The Non-analytic Case}

In this section, we consider the case that $\Omega$ satisfies the following two assumptions.

(1) $\Omega$ is a $C^{\infty}$ bounded domain;

(2) $\Omega\setminus\Gamma$ is piecewise analytic, where $\Gamma\subseteq\partial\Omega$ is a union of some $n-2$ dimensional submainfolds of $\partial\Omega$.

Because the method in the analytic case cannot be used here
directly, we adopt the argument developed in $\cite{Long Tian and
Xiaoping Yang}$ to deal with the non-analytic case. In this section,
we also assume that $\mu>0$.

We use $\partial\Omega(r)$ and $\Gamma(r)$ to denote the following
two sets respectively,
\begin{eqnarray*}
\partial\Omega(r)&=&\left\{x\in\overline{\Omega}|dist(x,\partial\Omega)\leq r\right\},\\
\Gamma(r)&=&\left\{x\in\overline{\Omega}|dist(x,\Gamma)\leq
r\right\}.
\end{eqnarray*}

First we need to do some preparation.

\begin{lemma}\label{monotonicity formula for the doubling index}
Let $u$ satisfy the equation $\triangle u+\lambda u=0$ in $B_r(x_0)$. Let $w(x,x_{n+1})=u(x)e^{\sqrt{\lambda}x_{n+1}}$. Then $w$ satisfies the equation $\triangle w=0$ on $\Omega\times\mathbb{R}$, i.e., $w$ is a harmonic function on $\Omega\times\mathbb{R}$. For the doubling index of $w$, which is denoted by $\bar{N}(x_0,r)$, it holds that for any $\epsilon>0$, there exists a positive constant $C$
depending only on $n$ and $\epsilon$, such that
\begin{equation}
\bar{N}(x_0,r)\leq (1+\epsilon)\bar{N}(x_0,2r)+C.
\end{equation}
\end{lemma}

\begin{proof}
Let $\Theta(x_0,r)=r\frac{\int_{B_r(x_0)}|\nabla w|^2dx}{\int_{\partial B_r(x_0)}w^2d\sigma}.$ Such quantity is called the frequency function of $w$, which can be seen in $\cite{Q.Han and F.H.Lin, F.H.Lin}$. $\Theta(x_0,r)$ is monotonicity to $r$, and has the following doubling conditions for any $t>1$.
\begin{equation*}
t^{\Theta(x_0,r/t)}\leq \left(\frac{\fint_{B_r(x_0)}w^2dx}{\fint_{B_{r/t}(x_0)}w^2dx}\right)^{\frac{1}{2}}\leq t^{\Theta(x_0,r)}.
\end{equation*}

From the standard interior estimate, we also have that for $\epsilon>0$,
\begin{equation*}
\sup\limits_{B_r(x_0)}|u|\leq C(\epsilon)\left(\fint_{B_{(1+\epsilon)r}(x_0)}u^2dx\right)^{\frac{1}{2}}
\end{equation*}
where $C(\epsilon)$ is a positive constant depending only on $\epsilon$. So for any $r>0$ and any $\epsilon_1>1$, we have that
\begin{eqnarray*}
\bar{N}(x_0,r)&=&\log_2\frac{\sup\limits_{B_r(x_0)}|u|}{\sup\limits_{B_{r/2}(x_0)}|u|}
\\&\leq&\log_2\frac{C(\epsilon_1)\left(\fint_{B_{(1+\epsilon_1)r}(x_0)}u^2dx\right)^{\frac{1}{2}}}
{\left(\fint_{B_{r/2}(x_0)}u^2dx\right)^{\frac{1}{2}}}
\\&\leq&\log_22(1+\epsilon_1)\Theta(x_0,(1+\epsilon_1)r)+C_1(\epsilon_1).
\end{eqnarray*}
We also have that
\begin{eqnarray*}
\bar{N}(x_0,r)&=&\log_2\frac{\sup\limits_{B_r(x_0)}|u|}{\sup\limits_{B_{r/2}(x_0)}|u|}
\\&\geq&\log_2\frac{\left(\fint_{B_{r}(x_0)}u^2dx\right)^{\frac{1}{2}}}
{C(\epsilon_1)\left(\fint_{B_{(1+\epsilon_1)r/2}(x_0)}u^2dx\right)^{\frac{1}{2}}}
\\&\geq&\log_2\left(\frac{2}{1+\epsilon_1}\right)\Theta(x_0,(1+\epsilon_1)r/2)-C_2(\epsilon_1).
\end{eqnarray*}
From the above inequalities, we have that
\begin{eqnarray*}
\bar{N}(x_0,r)&\leq&\log_22(1+\epsilon_1)\Theta(x_0,(1+\epsilon_1)r)+C_1(\epsilon)
\\&\leq&\frac{1+\log_2(1+\epsilon_1)}{1-\log_2(1+\epsilon_1)}N(x_0,2\frac{1+\log_2(1+\epsilon_1)}
{1+\log_2(1+\epsilon_1)}r)
+C_1(\epsilon_1)+C_2(\epsilon_1).
\end{eqnarray*}
For any given $\epsilon>0$, choose $\epsilon_1>0$ small enough such that $\frac{1+\log_2(1+\epsilon_1)}{1-\log_2(1+\epsilon_1)}\leq(1+\epsilon)$, we can get the desired result.
\end{proof}

\begin{remark}\label{relationship}
From some direct calculation, it is easy to know that
\begin{equation*}
N(x_0,r)\leq C(\bar{N}(x_0,2r)+\sqrt{\lambda}r),
\end{equation*}
where $C$ is a positive constant depending only on $n$.
\end{remark}

\begin{lemma}\label{preparation for the L2 norm}
Let $u$ be an eigenfunction
and $\lambda$ be the corresponding eigenvalue. Assume that $\lambda$ is large enough and $\mu>0$. Then, for
\begin{equation}
r_0=C\left(\sqrt{\lambda}+\mu\right)^{-\frac{n+4}{2n}}\frac{\sqrt{\mu}}{1+\sqrt{\mu}},
\end{equation}
it holds that
\begin{equation}
\|u\|_{L^2(\partial\Omega(r_0))}\leq\frac{1}{2}\|u\|_{L^2(\Omega)},
\end{equation}
where $C$ is a positive constant depending only on $n$ and $\Omega$.
\end{lemma}

\begin{proof}
By the standard elliptic estimate and the Sobolev Embedding Theorem,
we have
\begin{eqnarray*}
\|u\|_{L^2(\partial\Omega(r_0))}&\leq&Cr^n_0\|u\|_{L^{\infty}(\Omega)}
\\&\leq&Cr^n_0\left(\lambda^{\frac{n+2}{4}}\|u\|_{L^2(\Omega)}+\|u\|_{L^{\infty}(\partial\Omega)}\right).
\end{eqnarray*}
Because $\Omega$ is $C^{\infty}$ bounded,   $(\ref{eigen equation on
boundary})$ also holds on $\partial\Omega$. Thus by the standard
elliptic and the Sobolev Embedding Theorem again, we have
\begin{equation*}
\|u\|_{L^{\infty}(\partial\Omega)}\leq C\left(\sqrt{\lambda}+\mu\right)^{\frac{n+1}{2}}\|u\|_{L^2(\partial\Omega)}.
\end{equation*}
So
\begin{equation}\label{3.3}
\begin{array}{l}
\|u\|_{L^2(\partial\Omega(r_0))} \leq
Cr^n_0\left(\lambda^{\frac{n+2}{4}}\|u\|_{L^2(\Omega)}+(\sqrt{\lambda}+\mu)^{\frac{n+1}{2}}\|u\|_{L^2(\partial\Omega)}\right)
\\ \quad\quad\quad\quad\;\ \leq  Cr^n_0(\sqrt{\lambda}+\mu)^{\frac{n+2}{2}}\left(\|u\|_{L^2(\Omega)}+\|u\|_{L^2(\partial\Omega)}\right).
\end{array}
\end{equation}
By the Robin boundary condition, we have that
\begin{equation}\label{3.4}
\begin{array}{l}
\|u\|_{L^2(\partial\Omega)}^2 = \int_{\partial\Omega}u^2d\sigma
\\ \quad\quad\quad\;\;=\frac{1}{-\mu}\int_{\partial\Omega}uu_{\nu}d\sigma
\\ \quad\quad\quad\;\;= -\frac{1}{\mu}\int_{\Omega}div(u\nabla u)dx
\\ \quad\quad\quad\;\;= -\frac{1}{\mu}\int_{\Omega}|\nabla u|^2dx+\frac{\lambda}{\mu}\int_{\Omega}u^2dx
\\ \quad\quad\quad\;\;\leq \frac{\lambda}{\mu}\|u\|_{L^2(\Omega)}^2.
\end{array}
\end{equation}
By (\ref{3.3}) and  (\ref{3.4}), we have
\begin{equation*}
\|u\|_{L^2(\partial\Omega(r_0))}\leq
Cr^n_0\left(\sqrt{\lambda}+\mu\right)^{\frac{n+4}{2}}\left(1+\frac{1}{\sqrt{\mu}}\right)\|u\|_{L^2(\Omega)}.
\end{equation*}
So there exists some positive constant $C'$ depending only on $n$ and $\Omega$ such that for $r_0=C'(\sqrt{\lambda}+\mu)^{-\frac{n+4}{2n}}\frac{\sqrt{\mu}}{1+\sqrt{\mu}}$, the desired result holds.
\end{proof}

\begin{lemma}\label{extending again}
Let $u$ be an eigenfunction
and $\lambda$ be the corresponding eigenvalue. Assume that $\lambda$ is large enough and $\mu>0$. Then
for any $x\in\overline{\Omega}\setminus\Gamma$, it holds that
\begin{equation}
\|u\|_{L^{\infty}(B_r(x))}\leq e^{C(\sqrt{\lambda}+\mu-\ln\mu-\ln r)}\|u\|_{L^2(\Omega)},
\end{equation}
where $r=\frac{1}{2}\min\left\{dist(x,\Gamma),\delta\right\}$, $\delta$ is the same constant as in Lemma $\ref{extending}$, $C$ is a positive constant depending only on $n$ and $\Omega$.
\end{lemma}

\begin{proof}
 Without loss of generality, we may assume that $dist(x,\Gamma)\leq\delta$. Since $\partial\Omega\setminus\Gamma$ is piecewise analytic,
 all the derivation of $u$ on the whole domain $\Omega$ can be estimated by the same way as in Lemma $\ref{extending}$. Thus we have
\begin{eqnarray*}
\|u\|_{L^{\infty}(B_r(x))}&\leq&\sum\limits_{k=0}^{\infty}\sum\limits_{|\alpha|=k}\frac{r^k}{\alpha!}|D^{\alpha}u(x)|
\\&\leq&\sum\limits_{k=0}^{\infty}\sum\limits_{|\alpha|=k}
r^k\cdot\frac{C^{k+\frac{n+2}{2}}(\sqrt{\lambda}+\mu)^{k+\frac{n+2}{2}}}{r^{k+\frac{n+2}{2}}}
\left(\|u\|_{L^2(\Omega)}+\|u\|_{L^2(\partial\Omega)}\right)
\\&\leq&e^{C(\sqrt{\lambda}+\mu-\ln r)}\left(\|u\|_{L^2(\Omega)}+\|u\|_{L^2(\partial\Omega)}\right)
\\&\leq&e^{C(\sqrt{\lambda}+\mu-\ln
r)}(1+\frac{\sqrt{\lambda}}{\sqrt{\mu}})\|u\|_{L^2(\Omega)}.
\end{eqnarray*}
In the last inequality we have used the same arguments as in the
proof of Lemma $\ref{preparation for the L2 norm}$. Thus we get the
desired result since
$\sqrt{\lambda}/\sqrt{\mu}=e^{C(\ln\lambda-\ln\mu)}$ and we have
already required that the eigenvalue $\lambda$ is large enough.
\end{proof}

Now we will consider the upper bound for the doubling index of $u$ introduced in Section 2.

\begin{lemma}\label{upper bound for doubing index away gamma}
Let $u$ be an eigenfunction on $\Omega$
and $\lambda$ is the corresponding eigenvalue. Moreover, we also assume that $\mu>0$ and $\lambda$ is large enough. Then there exists a positive number $R_0$ depending only on $n$ and $\Omega$, such that for any $x\in\Omega\setminus\Gamma(R_0)$ and any $\alpha\in(0,1)$,
\begin{equation}
N(x,R_0)\leq C(\sqrt{\lambda}+\mu^{\alpha}+\mu^{-c\alpha}),
\end{equation}
where $C$ is a positive constant depending on $n$, $\Omega$ and $\alpha$, and $c$ is a positive constant depending only on $n$.
\end{lemma}

\begin{proof}
Because $\partial\Omega$ is bounded and $C^{\infty}$ smooth, there
exists some positive constant $R_0$ depending only on $n$ and
$\Omega$, such that for any point $x\in\partial\Omega(R_0)$, there
is one and only one point $x'\in\partial\Omega$ satisfying that
$dist(x,\partial\Omega)=dist(x,x')$.
Define $w(x,x_{n+1})=u(x)e^{\sqrt{\lambda}x_{n+1}}$.
Let $\overline{x}$ be the maximum point of $u$ on
$\overline{\Omega\setminus\partial\Omega(r_0)}$, where $r_0$ is the
same positive constant as in Lemma $\ref{preparation for the L2
norm}$. On the one hand, it is obvious that
\begin{equation*}
\|w\|_{L^{\infty}(B_{r_0}(\overline{x}))}\leq\|u\|_{L^{\infty}(\Omega)}e^{\sqrt{\lambda}r_0}.
\end{equation*}
On the other hand, by Lemma $\ref{preparation for the L2 norm}$
there holds
$\|u\|_{L^2(\Omega\setminus\partial\Omega(r_0))}\geq\frac{1}{2}\|u\|_{L^2(\Omega)}$,
and thus we have
\begin{equation}\label{3.7}
\begin{array}
{l} \|u\|_{L^{\infty}(B_{r_0/2}(\overline{x}))} \geq
|u(\overline{x})|
\\ \quad\quad\quad\quad\;\quad= \|u\|_{L^{\infty}(\Omega\setminus\partial\Omega(r_0))}\\  \quad\quad\quad\quad\quad\;\geq C\|u\|_{L^{2}(\Omega\setminus\partial\Omega(r_0))}
\\ \quad\quad\quad\quad\quad\;\geq  \frac{C}{2}\|u\|_{L^2(\Omega)}.
\end{array}
\end{equation}
Because
\begin{equation}\nonumber
\begin{array}
{l} \|u\|_{L^{\infty}(\Omega)}
 \leq C\left(\sqrt{\lambda}+\mu\right)^{\frac{n+2}{2}}\left(\|u\|_{L^2(\Omega)}+\|u\|_{L^2(\partial\Omega)}\right)
\\ \quad\quad\quad\; \leq C\left(\sqrt{\lambda}+\mu\right)^{\frac{n+2}{2}}\left(1+\frac{\sqrt{\lambda}}{\sqrt{\mu}}\right)\|u\|_{L^2(\Omega)}
\\ \quad\quad\;\quad\leq C\left(\sqrt{\lambda}+\mu\right)^{\frac{n+4}{2}}/\sqrt{\mu}\|u\|_{L^2(\Omega)},
\end{array}
\end{equation}
it holds that
\begin{equation}\label{3.8}
\|u\|_{L^2(\Omega)}\geq
C\sqrt{\mu}\left(\sqrt{\lambda}+\mu\right)^{-\frac{n+4}{2}}\|u\|_{L^{\infty}(\Omega)}.
\end{equation}
Then by (\ref{3.7}) and (\ref{3.8}), we have
\begin{equation*}
\|w\|_{L^{\infty}(B_{r_0/2}(\overline{x}))}\geq\|u\|_{L^{\infty}(B_{r_0/2}(\overline{x}))}\geq
C\sqrt{\mu}\left(\sqrt{\lambda}+\mu\right)^{-\frac{n+4}{2}}\|u\|_{L^{\infty}(\Omega)}.
\end{equation*}
Therefore
\begin{eqnarray*}
\bar{N}(\overline{x},r_0)&=&\log_2\frac{\sup\limits_{B_{r_0}(\overline{x},0)}|w|}
{\sup\limits_{B_{r_0/2}(\overline{x},0)}|w|}
\\&\leq&\log_2\frac{\|u\|_{L^{\infty}(\Omega)}}{C\sqrt{\mu}(\sqrt{\lambda}+\mu)^{-\frac{n+4}{2}}
\|u\|_{L^{\infty}(\Omega)}}+C\sqrt{\lambda}r_0\\&
\leq& C(\ln(\sqrt{\lambda}+\mu)-\ln\mu+\sqrt{\lambda}r_0).
\end{eqnarray*}
By Lemma $\ref{monotonicity formula for the doubling index}$ with $\epsilon$ satisfies that $\log_{\frac{5}{4}}(1+\epsilon)\cdot\frac{2n}{n+4}=\frac{\alpha}{4}$,  it holds that
\begin{equation*}
\bar{N}(\overline{x},r_0/2)\leq C(1+\epsilon)(\ln(\sqrt{\lambda}+\mu)-\ln\mu+\sqrt{\lambda}r_0),
\end{equation*}
provided that $\lambda$ is large enough.

Let $x_1\in \partial
B_{r_0/4}(\overline{x})\cap\left(\Omega\setminus\partial\Omega(r_0)\right)$
and let $r_1=\frac{5}{4}r_0$. Then by the triangle inequality, there
holds  $dist(x_1,\partial\Omega)\leq r_1$. Thus
\begin{equation*}
\|w\|_{L^{\infty}(B_{r_1}(x_1))}\leq\|u\|_{L^{\infty}(\Omega)}e^{\sqrt{\lambda}r_1}.
\end{equation*}
On the other hand, we have
\begin{eqnarray*}
\|w\|_{L^{\infty}(B_{r_1/2}(x_1))}&\geq&\|w\|_{L^{\infty}(B_{r_0/4}(\overline{x}))}
\\&\geq&2^{-\bar{N}(\overline{x},r_0/2)}\|w\|_{L^{\infty}(B_{r_0/2}(\overline{x}))}
\\&=&2^{-C(\ln(\sqrt{\lambda}+\mu)-\ln\mu+\sqrt{\lambda}r_0)}
\sqrt{\mu}(\sqrt{\lambda}+\mu)^{-\frac{n+4}{2}}\|u\|_{L^{\infty}(\Omega)}
\\&=&2^{-C(\frac{n+6}{2}\ln(\sqrt{\lambda}+\mu)-\frac{3}{2}\ln\mu+\sqrt{\lambda}r_0)}\|u\|_{L^{\infty}(\Omega)}.
\end{eqnarray*}
Thus
\begin{equation*}
\bar{N}(x_1,r_1)\leq C(1+\epsilon)\left(\frac{n+6}{2}\ln(\sqrt{\lambda}+\mu)-\frac{3}{2}\ln\mu+\sqrt{\lambda}r_0\right)+C\sqrt{\lambda}r_1,
\end{equation*}
and then for any $r\leq r_1$,
\begin{equation*}
\bar{N}(x_1,r)\leq
C(1+\epsilon)^2\left(\frac{n+6}{2}\ln(\sqrt{\lambda}+\mu)-\frac{3}{2}\ln\mu+\sqrt{\lambda}r_0\right)
+C(1+\epsilon)\sqrt{\lambda}r_1.
\end{equation*}
Using the above arguments for $k$ times, such that
$(5/4)^{k-1}r_0<2R_0$ and $(5/4)^kr_0\geq 2R_0$, then
$k=C\ln(2R_0/r_0)=C(\ln(\sqrt{\lambda}+\mu)-\ln\mu)$. Then for some
point $x_k\in\Omega$ with
\begin{equation*}
dist(x_k,\partial\Omega)\geq r_k \geq R_0,
\end{equation*}
it holds that
\begin{eqnarray*}
\bar{N}(x_k,r_k)&\leq&C(1+\epsilon)^k\left(\frac{n+6}{2}\ln(\sqrt{\lambda}+\mu)-\frac{3}{2}\ln\mu\right)
\\&+&C\sqrt{\lambda}\left((1+\epsilon)^kr_0+(1+\epsilon)^{k-1}r_1+\cdots+(1+\epsilon)^{0}R_0\right)
\\&\leq&C\left(\frac{5}{4}\right)^{k\log_{5/4}(1+\epsilon)}\left(\ln(\sqrt{\lambda}+\mu)-\ln\mu\right)
\\&+&C\sqrt{\lambda}(1+\epsilon)^kr_0\left(1+\frac{5}{4(1+\epsilon)}+\left(\frac{5}{4(1+\epsilon)}\right)^2+\cdot
+\left(\frac{5}{4(1+\epsilon)}\right)^k\right)
\\&\leq&Cr_0^{-\log_{5/4}(1+\epsilon)}\left(\ln(\sqrt{\lambda}+\mu)-\ln\mu\right)+C\sqrt{\lambda}(1+\epsilon)^{k}r_0
\left(\frac{5}{4(1+\epsilon)}\right)^k
\\&\leq&C\left(\sqrt{\lambda}+\mu^{\alpha}+\mu^{-c\alpha}\right),
\end{eqnarray*}
where $c$ is a positive constant depending only on $n$.
In the above inequalities, $C$ may be different from line to line. Then for some point $\widetilde{x}\in \Omega\setminus\Gamma(R_0)$,
we have $N(\widetilde{x},2R_0)\leq
C\left(\sqrt{\lambda}+\mu+\ln^2\mu\right).$

For any point $x\in\Omega\setminus\Gamma(R_0)$, using the same
argument for $l$ times, but keeping the radius unchanged, where $l$
is a positive constant depending only on $n$ and $\Omega$, we can
obtain that
\begin{equation*}
N(x,R_0)\leq C\left(\sqrt{\lambda}+\mu^{\alpha}+\frac{1}{\mu^{c\alpha}}\right),
\end{equation*}
which is the desired result.
\end{proof}

\begin{lemma}\label{upper bound for doubling index near gamma}
Let $u$ be an eigenfunction on $\Omega$
and $\lambda$ is the corresponding eigenvalue. Moreover, we also assume that $\mu>0$ and $\lambda$ is large enough. Then for any $x\in\Gamma(R_0)\setminus\Gamma$ and $\overline{r}<dist(x,\Gamma)/2$, it holds that for any $\alpha\in(0,1)$,
\begin{equation}
N(x,\overline{r})\leq C\overline{r}^{-\frac{1}{2}}\left(\sqrt{\lambda}+\mu^{\alpha}+\mu^{-c\alpha}\right),
\end{equation}
where $R_0$ and $c$ are the same positive constants as in Lemma $\ref{upper bound for doubing index away gamma}$, $C$ is a positive constant depending on $n$, $\Omega$ and $\alpha$.
\end{lemma}

\begin{proof}
Also let $w(x,x_{n+1})=u(x)e^{\sqrt{\lambda}x_{n+1}}$.
By Lemma $\ref{upper bound for doubing index away
gamma}$ and Lemma $\ref{monotonicity formula for the doubling index}$, we know that for any
$x_0\in\Omega\setminus\partial\Omega(R_0)$ and $r<R_0$, there holds
\begin{equation*}
\bar{N}(x_0,r)\leq
C(1+\epsilon)\left(\sqrt{\lambda}+\mu^{\alpha}+\mu^{-c\alpha}\right).
\end{equation*}
Choose $\epsilon>0$ satisfies that $\log_{3/4}(1+\epsilon)=-\frac{1}{2}$.

Because
\begin{equation*}
\|w\|_{L^{\infty}(B_{R_0/16}(x_0))}\geq\|u\|_{L^{\infty}(B_{R_0/16}(x_0))}\geq 2^{-C(\sqrt{\lambda}+\mu^{\alpha}+\mu^{-c\alpha})}\|u\|_{L^{\infty}(\Omega)}.
\end{equation*}
So for any $x_1\in B_{R_0/4}(x_0)\cap\partial\Omega(R_0)$, it holds
that $dist(x_1,\partial\Omega)\leq\frac{3}{4}R_0$. Let
$R_1=\frac{3}{4}R_0$, then
\begin{equation*}
\|w\|_{L^{\infty}(B_{R_1}(x_1))}\leq\|u\|_{L^{\infty}(\Omega)}e^{\sqrt{\lambda}R_1}
\end{equation*}
and
\begin{equation*}
\|w\|_{L^{\infty}(B_{R_1/2}(x_1))}\geq\|w\|_{L^{\infty}(B_{R_0/16}(x_0))}\geq 2^{-C(\sqrt{\lambda}+\mu^{\alpha}+\frac{1}{\mu^{c\alpha}})}\|u\|_{L^{\infty}(\Omega)}.
\end{equation*}
So
\begin{equation*}
\bar{N}(x_1,R_1)\leq
C(1+\epsilon)\left(\sqrt{\lambda}+\mu^{\alpha}+\mu^{-c\alpha}\right)+C\sqrt{\lambda}R_1,
\end{equation*}
and furthermore for any $r\leq R_1$, we have
\begin{equation*}
\bar{N}(x_1,R_1/2)\leq
C(1+\epsilon)^2\left(\sqrt{\lambda}+\mu^{\alpha}+\mu^{-c\alpha}\right)+C(1+\epsilon)\sqrt{\lambda}R_1,
\end{equation*}
and
\begin{equation*}
\|w\|_{L^{\infty}(B_{R_1/16}(x_1))}\geq 2^{-C(1+\epsilon)^2(\sqrt{\lambda}+\mu^{\alpha}+\mu^{-c\alpha})-C(1+\epsilon)\sqrt{\lambda}R_1)}
\|u\|_{L^{\infty}(\Omega)}.
\end{equation*}
By doing this $k$ times such that
$R_k=(3/4)^{k-1}R_0>2\overline{r}$ and $R_k=(3/4)^kR_0\leq2\overline{r}$,
i.e., $k=-C\ln\overline{r}$, we have
\begin{eqnarray*}
\bar{N}(x_k,R_k)&\leq&C(1+\epsilon)^k\left(\sqrt{\lambda}+\mu^{\alpha}+\mu^{-c\alpha}\right)
\\&+&C\sqrt{\lambda}\left((1+\epsilon)^kR_0+(1+\epsilon)^{k-1}R_1+\cdots+(1+\epsilon)^0R_k)\right)
\\&\leq&C\overline{r}^{-\frac{1}{2}}\left(\sqrt{\lambda}+\mu^{\alpha}+\mu^{-c\alpha}\right)
\\&+&C\sqrt{\lambda}(1+\epsilon)^kR_0\left(1+\frac{3}{4(1+\epsilon)}+\left(\frac{3}{4(1+\epsilon)}\right)^2
+\cdots+\left(\frac{3}{4(1+\epsilon)}\right)^k\right)
\\&\leq&C\overline{r}^{-\frac{1}{2}}\left(\sqrt{\lambda}+\mu^{\alpha}+\mu^{-c\alpha}\right)
+C\sqrt{\lambda}R_0(1+\epsilon)^k
\\&\leq&C\overline{r}^{-\frac{1}{2}}\left(\sqrt{\lambda}+\mu^{\alpha}+\mu^{-c\alpha}\right).
\end{eqnarray*}
for any $r<\overline{r}$. Then by repeating the same argument for
$l$ times, where $l$ depends only on $n$ and $\Omega$, and keeping
the radius unchanged, we have that, for any
$x\in\Omega\setminus\Gamma(2\overline{r})$, the
following inequality holds:
\begin{equation*}
\bar{N}(x,2\overline{r})\leq
C\overline{r}^{-\frac{1}{2}}\left(\sqrt{\lambda}+\mu^{\alpha}+\mu^{-c\alpha}\right),
\end{equation*}
where $C$ is a positive constant depending only on $n$, $\Omega$ and $\alpha$ and $c$ is a positive constant depending only on $n$, provided that $\lambda$ is large enough.
Then from Remark $\ref{relationship}$, we can get the desired result.
\end{proof}

With the above preparation, we can get the conclusion of Theorem
$\ref{another main theorem}$.

\textbf{Proof of Theorem \ref{another main theorem}:}

By Lemmas $\ref{nodal estimate in small ball}$,  $\ref{upper bound
for doubing index away gamma}$ and $\ref{upper bound for doubling
index near gamma}$, we have that
\begin{equation}\label{3.10}
\begin{array}
{l} \mathcal{H}^{n-1}\left(\left\{x\in\Omega|u(x)=0\right\}\right)
 \leq \mathcal{H}^{n-1}\left(\left\{x\in\Omega\setminus\Gamma(R_0)|u(x)=0\right\}\right)
\\ \quad\quad\quad\quad\quad\quad\quad\quad\quad\quad + \sum\limits_{k=0}^{\infty}\mathcal{H}^{n-1}
\left(\left\{x\in\Gamma(\frac{R_0}{2^k})\setminus\Gamma(\frac{R_0}{2^{k+1}})|u(x)=0\right\}\right)
\\ \quad\quad\quad\quad\quad\quad\quad\quad\quad\quad+ \mathcal{H}^{n-1}\left(\left\{x\in\Gamma|u(x)=0\right\}\right).
\end{array}
\end{equation}
By Lemma $\ref{nodal estimate in small ball}$ and Lemma $\ref{upper
bound for doubing index away gamma}$, there holds
\begin{equation}\label{3.11}
\begin{array}{l}
\mathcal{H}^{n-1}\left(\left\{x\in\Omega\setminus\Gamma(R_0)|u(x=0)\right\}\right)
 \leq
C\left(\sqrt{\lambda}+\mu^{\alpha}+\mu^{-c\alpha}\right)R_0^{n-1}\frac{1}{R_0^n}
\\ \quad\quad\quad\quad\quad\quad\quad\quad\quad\quad\quad\quad\quad\leq C\left(\sqrt{\lambda}+\mu^{\alpha}+\mu^{-c\alpha}\right).
\end{array}
\end{equation}
By Lemma $\ref{nodal estimate in small ball}$ and Lemma $\ref{upper
bound for doubling index near gamma}$, we have that
\begin{equation}\label{3.12}
\begin{array}{l}
\sum\limits_{k=0}^{\infty}\mathcal{H}^{n-1}
\left(\left\{x\in\Gamma(\frac{R_0}{2^k})\setminus\Gamma(\frac{R_0}{2^{k+1}})|u(x)=0\right\}\right)
 \leq  \sum\limits_{k=0}^{\infty}C\left(\sqrt{\lambda}+\mu^{\alpha}+\mu^{-c\alpha}\right)\left(\frac{R_0}
 {2^{k}}\right)^{-\frac{1}{2}}
\\ \quad\quad\quad\quad\quad\quad\quad\quad\quad\quad\quad\quad\quad\quad\quad\quad\times
\left(\frac{R_0}{2^k}\right)^{n-1}\frac{1}{\left(\frac{R_0}{2^k}\right)^{n-2}}
\\ \quad\quad\quad\quad\quad\quad\quad\quad\quad\quad\quad\quad\quad\quad\quad\quad\leq C\sum\limits_{k=0}^{\infty}\frac{1}{\sqrt{2}^k}(\sqrt{\lambda}+\mu^{\alpha}+\mu^{-c\alpha})
\\\quad\quad\quad\quad\quad\quad\quad\quad\quad\quad\quad\quad\quad\quad\quad\quad \leq C(\sqrt{\lambda}+\mu^{\alpha}+\mu^{-c\alpha}).
\end{array}
\end{equation}
Because $\Gamma$ is a union of some $(n-2) $ dimensional
submanifolds of $\partial\Omega$, we have
\begin{equation}\label{3.13}
\mathcal{H}^{n-1}\left(\left\{x\in\Gamma|u(x)=0\right\}\right)\leq\mathcal{H}^{n-1}\left(\Gamma\right)
=0.
\end{equation}
Substituting (\ref{3.11})-(\ref{3.13}) into (\ref{3.10}), we have
that
\begin{equation*}
\mathcal{H}^{n-1}\left(\left\{x\in\Omega|u(x)=0\right\}\right)\leq
C\left(\sqrt{\lambda}+\mu^{\alpha}+\mu^{-c\alpha}\right),
\end{equation*}
where $C$ is a positive constant depending only on $n$, $\Omega$, $\Gamma$ and $\alpha$ and $c$ is a positive constant depending only on $n$, provided that $\lambda$ is large enough. That is the desired result.

\qed


\newpage

\begin{thebibliography}{}

\bibitem{Almgren} F.J. Almgren Jr., Dirichlet's problem for multiple valued functions and the regularity of mass minimizing integral currents, M.Obata(Ed.),
Minimal Submanifolds and Geodesics, North-Holland, Amsterdam,
(1979), 1-6.


\bibitem{S.Ariturk} S. Ariturk, Lower bounds for nodal sets of Dirichlet and Neumann eigenfunctions,
Comm. Math. Physics, 317, 3, (2013), 817-825.

\bibitem{K.Bellova and F.H.Lin} K. Bellov$\acute{a}$, F.H. Lin, Nodal sets of Steklov eigenfunctions,
Calculus of Variations and Partial Differential Equations, 54, 2,
(2015), 2239-2268.

\bibitem{J.Bruning} J. Br$\ddot{u}$ning, $\ddot{U}$ber Knoten von eigenfunktionen des Laplace-Beltrami operators,
Math. Z., 158,1, (1978),  15-21.

\bibitem{J.E.Chang} J.E. Chang, Lower bounds for nodal sets of biharmonic Steklov problems, J. London Math. Society,
95, 3, (2017), 763-784.

\bibitem{T.H.Colding and W.P.Minicozzi} T. H. Colding, W. P. Minicozzi II,
Lower bounds for nodal sets of eigenfunctions, Comm. Math. Physics,
306, 3, (2011), 777-784.

\bibitem{R.T.Dong} R.T. Dong, Nodal sets of eigenfunctions on Riemannian surfaces,
J. Differential Geometry, 36, (1992), 493-506.

\bibitem{Donnelly and Fefferman} H. Donnelly, C. Fefferman, Nodal sets of eigenfunctions on Riemannian manifolds, Invent.Math., 93, (1988), 161-183.

\bibitem{H.Federer} H. Federer, Geomegric measure theory,
Springer, Verlag, New York, (1969).

\bibitem{N.Garofalo and F.H.Lin} N. Garfalo, F.H. Lin, Unique continuation for elliptic operators: a geometric-variational approach, Comm. Pure. Appl. Math., 40, (1987), 347-366.


\bibitem{D.Gilbarg and N.S.Trudinger} D. Gilbarg, N.S. Trudinger,
Elliptic partial differential equation of second order,
Springer, Berlin, (1983).

\bibitem{Q.Han R.Hardt F.H.Lin} Q. Han, R. Hardt, F.H. Lin, Geometric measure of singular sets of elliptic equations,
Comm. Pure. Appl. Math., 51, (1998), 1425-1443.

\bibitem{Q.Han F.H.Lin} Q. Han, F.H. Lin, On the geometric measure of nodal sets of solutions, J.Part. Diff. Equ., 7,
(1994), 111-131.

\bibitem{Q.Han and F.H.Lin} Q. Han, F.H. Lin, Nodal sets of solutions of elliptic differential equations,
 Book in preparation, available at
 http://www.nd.edu/~qhan/nodal.pdf, (2007).

\bibitem{R.Hardt and L.Simon} R. Hardt, L. Simon, Nodal sets for solutions of elliptic equations,
J. Differential Geometry, 30, (1989), 505-522.

\bibitem{I.Kukavica} I. Kukavica, Nodal volumes for eigenfunctions of analytic regular elliptic problems,
J. d'Analyse Math$\acute{e}$matique, 67, 1, (1995), 269-280.

\bibitem{F.H.Lin} F.H. Lin, Nodal sets of solutions of elliptic and parabolic equations,
Comm. Pure Appl. Math., 44, (1991), 287-308.

\bibitem{F.H.Lin and X.P.Yang} F.H. Lin, X.P. Yang, Geometric measure theory-an introduction,
Adv. Math., vol. 1, Science Press/International Press, Beijing/Boston, (2002).

\bibitem{A.Logunov} A. Logunov, Nodal sets of Laplace eigenfunctions: polynomial upper estimates of the Hausdorff measure, Annals of Mathematics, 187, (2018), 221-239.

\bibitem{A.Logunov2} A. Logunov, Nodal sets of Laplace eigenfunctions: proof of Nadirashvili's conjecture and of the lower bound in Yau's conjecture,  Annals of Mathematics, 187, (2018), 241-262.

\bibitem{A.Logunov and E.malinnikova} A. Logunov, E. Malinnikova, Nodal sets of Laplace eigenfunctions: estimates of the Hausdorff measure in dimension two and three, arXiv: 1605.02595, (2016).

\bibitem{C.D.Sogge and S.Zelditch} C.D. Sogge, S. Zelditch,
Lower bounds on the Hausdorff measure of nodal sets, Math. Res.
Lett., 18, 1, (2011), 25-37.

\bibitem{Long Tian and Xiaoping Yang} L. Tian, X.P. Yang, Measure Upper Bounds for Nodal Sets of Eigenfunctions of the bi-Harmonic Operator, arXiv: 1709.00153, (2017).

\bibitem{Long Tian and Xiaoping Yang2} L. Tian, X.P. Yang, Nodal sets and horizontal singular sets of H-harmonic functions on the Heisenberg group, Commun. Contemp. Math., 16, 4, (2014).

\bibitem{S.T.Yau} S.T. Yau, Open problems in geometry,
Proc. Sympos.  Pure Math., 54, 1, (1993), 1-28.

\bibitem{J.Y.Zhu} J.Y. Zhu, Interior nodal sets of Steklov eigenfunctions on
surfaces, Anal. PDE, 9, 4, (2016), 859-880.




\end{thebibliography}
\end{document}